\documentclass[a4paper,12pt]{amsart}
\usepackage{fullpage,amscd,amssymb,amsfonts,amsmath,amsthm}
\usepackage{hyperref}
\newtheorem*{thm*}{Theorem}
\newtheorem{thm}{Theorem}[section]
\newtheorem{prop}[thm]{Proposition}
\newtheorem{lemma}[thm]{Lemma}

\theoremstyle{definition}

\newtheorem{rem}[thm]{Remark}

\newcommand{\Cb}{{\mathbb C}}

\newcommand{\Pb}{{\mathbb P}}
\newcommand{\Rb}{{\mathbb R}}

\newcommand{\Zb}{{\mathbb Z}}

\newcommand{\rank}{{\rm rank}}

\newcommand{\Hom}{\textrm{Hom}}

\newcommand{\Nef}{\N{E}{F}}

\theoremstyle{definition}

\renewcommand{\Nef}{\text{Nef}}

\newcommand{\Psef}{\text{Psef}}
\newcommand{\NE}{\overline{\text{NE}}}
\usepackage{fancyhdr}
\usepackage{lipsum}

\begin{document}
\title{Some surfaces with non-polyhedral nef cones}
\author{Ashwath Rabindranath}
\address{University of Michigan\\ Department of Mathematics\\ Ann Arbor, MI48109\\USA}
\email{ashwathr@umich.edu}
\begin{abstract} We study the nef cones of complex smooth projective surfaces and give a sufficient criterion for them to be non-polyhedral. We use this to show that the nef cone of $C \times C$, where $C$ is a complex smooth projective curve of genus at least $2$, is not polyhedral. \end{abstract}
\maketitle
\section{Introduction}
\noindent There has been a great deal of interest in understanding the various positive cones of curves and divisors on algebraic varieties. Several cases have been analyzed, including symmetric products of curves in \cite{Kou}, \cite{Pac}, abelian varieties in \cite{Bau}, \cite{DELV}, and holomorphic symplectic varieties in \cite{HT}, \cite{BHT}. The main result of this paper is the following theorem.

\begin{thm}\label{thm1} If $C$ is a smooth projective curve over $\Cb$ of genus $g \geq 2$, the nef cone of $C \times C$ is not polyhedral. \end{thm}

\noindent We address the cases when $g < 2$. If $C$ has genus $0$, it is isomorphic to $\Pb^1$. The nef cone of $\Pb^1 \times \Pb^1$ is rational polyhedral and is equal to $$\{ (x, y) \in \Rb^2 : x \geq 0, y \geq 0 \}.$$ If $C$ is a curve of  genus $1$ and $h$ is an ample class on $C$, the nef cone of $C \times C$ is precisely $$\{ \alpha \in N^1(C \times C)_\Rb : (\alpha \cdot \alpha) \geq 0, (\alpha \cdot h) \geq 0 \}.$$ In this case, the nef cone is not polyhedral. 
\vspace{0.1in}

\noindent In section \ref{sec2}, we prove a sufficient criterion for the nef cone of a surface to be non-polyhedral. In section \ref{sec3}, we use this criterion to prove that the nef cone of $C \times C$ is not polyhedral for $C$ a complex smooth projective curve of genus at least $2$.

\subsection*{Acknowledgements} I am grateful to my advisor Mircea Musta\c{t}\u{a} for his support and guidance. I also thank Harold Blum, Felipe Perez and Phil Tosteson for useful conversations. 
\section{Criterion for non-polyhedral nef cones}\label{sec2}

\noindent In this section we prove a sufficient criterion for nef cones to not be polyhedral. We begin by fixing some notation. For any smooth projective variety $X$, we denote by $N^1(X)$ the free and finitely generated $\Zb$-module of numerical equivalence classes of divisors on $X$. Let $\rho(X)$ be its rank. We use $\equiv$ to denote numerical equivalence. Let $N^1(X)_\Rb:= N^1(X) \otimes_\Zb \Rb$. The closed convex cone generated by numerical classes of nef divisors is the \textit{nef cone}, denoted by $\text{Nef}(X)$. The closed convex cone generated by numerical classes of effective divisors is the \textit{pseudoeffective cone} denoted by $\Psef(X)$. 
\vspace{0.1in} 

\noindent In what follows, we assume that $X$ is a smooth projective surface. For such $X$, $N^1(X)$ is equipped with the usual intersection form and  $\Psef(X)$ is the same as the \textit{Mori cone} (denoted by $\NE(X)$) which is the dual of the nef cone under the intersection product. Recall that a cone $\sigma$ is said to be \textit{polyhedral} if it is the positive span of finitely many vectors. A theorem of Farkas (\cite{F}, Pg. 11) tells us that a cone $\sigma$ is polyhedral if and only if $\sigma^\vee$ is polyhedral. 
\vspace{0.1in}

\noindent Suppose $\rho(X) \geq 3$ and pick an orthogonal basis $\{h, f_1, \ldots, f_{\rho(X) - 1}\}$ of $N^1(X)_\Rb$ such that $h$ is ample, $(h \cdot h) = 1$ and $(f_i \cdot f_i) = -1$ for $1 \leq i \leq \rho(X) - 1$. The existence of such a basis follows from the Hodge index theorem.

\begin{prop}\label{crit} For $X$ as above, if there exist $e$ and $f$ such that,
\begin{enumerate}
\item\label{a} $0 \neq e$ is a boundary class of $\NE(X)$ such that $(e \cdot e) = 0$,
\item\label{b} $0 \neq f$ is a class in the linear span of $\{f_1, \ldots, f_{\rho(X) - 1} \}$ such that $(e \cdot f) = 0$ and $(e + \Rb f) \cap \NE(X) = \{e\},$
\end{enumerate} then $\Nef(X)$ is not polyhedral.
\end{prop}

\begin{proof} Consider the lines  $\ell_1 := \{e +  sf : s \in \Rb\}$ and $\displaystyle \ell_2 := \{ te + (1-t)(h \cdot e)h : t \in \Rb \}$. These lines are distinct because otherwise $e + sf$ would equal  $h$ for some value of $s$, which is impossible since $(e + \Rb f) \cap \NE(X) = \{e\}$. The affine $2$-plane $P$ spanned by $\ell_1$ and $\ell_2$ is contained in the affine hyperplane $$H:= \{ v \in N^1(X) : (v \cdot h) = (e \cdot h) \}.$$ Since $0 \neq e \in \NE(X)$, we know that $(e \cdot h) > 0$ by Kleiman's criterion. The image of $\NE(X) \backslash \{0\}$ in $\Pb(N^1(X)_\Rb)$ is closed hence compact. Since $H$ maps homeomorphically onto its image in $\Pb(N^1(X)_\Rb)$, we conclude that $H \cap \NE(X)$ is compact. It follows that $P \cap \NE(X)$ is compact, being a closed subset of $H \cap \NE(X)$. 
\vspace{0.1in}

\noindent Assume that $\NE(X)$ is a polyhedral cone. It follows that $P \cap \NE(X)$ must be a convex polygon. Since $\ell_1$ intersects this convex polygon at precisely one point, $e$ must be a vertex. The class $h' := (e \cdot h)h$ lies in the interior of this polygon, being an ample class. Since $(e + \Rb f) \cap \NE(X) = \{e\}$, neither edge of the polygon emanating from $e$ is contained in $(e + \Rb f)$. Hence, $(h' + \Rb f)$ is not parallel to either of these edges and it must intersect both edges at precisely one point each, say $h' + \chi_if$ for $i = 1,2$. Picking $m >  \max(|\chi_1|, |\chi_2|)$ we see that the segment $\ell_3$ joining $e$ and $h' + mf$ lies entirely outside $\NE(X)$, aside from $e$. A general point on this segment is $$P_t := te + (1-t)(h' + mf) \text{ for } 0 \leq t \leq 1.$$ We compute the self-intersection \begin{align*} (P_t \cdot P_t) &= t^2(e \cdot e) + (1-t)^2(h' \cdot h') + (1-t)^2m^2(f \cdot f) + 2t(1-t)(e \cdot h') \\ &=  (1-t)( (1-t)(h' \cdot h') + (1-t)m^2(f \cdot f) + 2t(e \cdot h')) \end{align*} The term 
$(1-t)(h' \cdot h') + (1-t)m^2(f \cdot f) + 2t(e \cdot h'))$ is positive for $t = 1$ since $(e \cdot h') > 0$ because $e$ is pseudoeffective and $h'$ is ample. Hence for $t$ slightly less than $1$, this term is positive forcing $(P_t \cdot P_t)$ to be positive. Now this implies that either $P_t$ or $-P_t$ is big. Since $(P_t \cdot h) = (e \cdot h) > 0$, it follows that $P_t$ is big and contained in the interior of $\NE(X)$, a contradiction! We thus conclude that $\NE(X)$ is not polyhedral, hence $\Nef(X)$ is not polyhedral as well. \end{proof}

\section{Nef cone of $C \times C$}\label{sec3}

\noindent For the remainder of this paper, we focus on a fixed complex smooth projective curve $C$ of genus $g \geq 2$  Let $\Delta \subset C \times C$ be the diagonal and let $J$ be the Jacobian of $C$. Let $p_1, p_2 : C \times C \to  C$ be the projection morphisms.  Let $e_i$ be the numerical class of a fiber of $p_i$ and $\delta := \Delta - e_1 - e_2$. Recall  (see \cite{Rob}, Section 1.5) that \begin{align}\label{facts} (e_1 \cdot e_1) = (e_2 \cdot e_2)  = (e_1 \cdot \delta) = (e_2 \cdot \delta) = 0, \text{ }(e_1 \cdot e_2) = 1, \text{ and  }  (\delta \cdot \delta) = -2g. \end{align} Furthermore, we have \begin{align*} N^1(C \times C) &= p_1^*N^1(C) \oplus p_2^*N^1(C) \oplus \Hom(J, J) \\  & = \Zb \oplus \Zb \oplus \Hom(J, J). \end{align*} Since $\rank_\Zb(\Hom(J,J)) \geq 1$, it follows that $\rho(C \times C) \geq 3$. It is well known that the Mori cone is a full-dimensional cone in $N^1(C \times C)_\Rb$. 

\begin{lemma}\label{tang} For $\nu \in \NE(C \times C)$, we have $(e_2 \cdot \nu) \geq 0$. 
 \end{lemma} 
\begin{proof} This is immediate since $e_2 \equiv C \times \{P\}$ and is nef, hence is nonnegative on $\NE(C \times C)$.
\end{proof}

\noindent We need the following result of Vojta.
\begin{prop}[Proposition 1.5, \cite{vojta}]\label{vojta} Let $Y(r,s) := a_1e_1 + a_2e_2 + a_3\delta$ where $a_1 = \displaystyle \sqrt{\frac{g+s}{r}}$, $a_2 = \sqrt{(g+s)r}$ and $a_3 = \pm 1$, for $r, s \in \Rb_{> 0}$. If \begin{align*} r > \frac{(g + s)(g - 1)}{s}, \end{align*} then $Y(r,s)$ is nef. \end{prop}

\noindent In his paper, Vojta only considers the case $a_3 = 1$. For completeness, we sketch (with suitable modifications) the proof of Proposition \ref{vojta} below. 

\begin{proof}[Proof due to Vojta]  Assume, arguing by contradiction, that there exists a curve $C_0$ (not necessarily smooth) on $C \times C$ such that $(C_0 \cdot Y(r,s)) < 0$. We may assume that $C_0$ is irreducible. Note that it is not a fiber of $p_i$ for $i = 1,2$ since $(e_i \cdot Y(r,s)) \geq 0$. Applying the adjunction formula, we get \begin{align*} (C_0^2) + (2g-2)((C_0 \cdot e_1) + (C_0 \cdot e_2)) & =  (C_0^2) + (C_0 \cdot K_{C \times C}) \\ & = 2p_a(C_0) - 2  \\ & \geq 2p_g(C_0) - 2 \\ & \geq (2g - 2)(C_0 \cdot e_1), \end{align*} where $p_a(C_0)$ and $p_g(C_0)$ are the arithmetic and geometric genera\footnote{Recall that the geometric genus of a singular curve is defined as the genus of its normalization.} of $C_0$. Note that the last inequality follows by applying Riemann-Hurwitz to $p_1 \circ \eta : \widetilde{C_0} \to C$, where $\eta : \widetilde{C_0} \to C_0$ is the normalization. The composition $p_1 \circ \eta$ is a finite morphism because $C_0$ is not a fiber of either projection. We can then conclude that \begin{align}\label{eve} (C_0^2) + (2g-2)(C_0 \cdot e_2) \geq 0.\end{align}

\noindent Write $C_0 \equiv b_0\delta + b_1e_1 + b_2e_2 + \nu$ where $\nu$ is orthogonal to $\delta, e_1$ and $e_2$ in $N^1(C \times C)_\Rb$. The Hodge index theorem forces $(\nu \cdot \nu) \leq 0$. Using this and (\ref{eve}), we compute $$2b_1b_2 + (2g - 2)b_1 \geq 2gb_0^2.$$ Since $b_1 \geq 0$ and is an integer (being equal to $(C_0 \cdot e_2)$) we have $b_1^2 \geq b_1$ and can write \begin{align}\label{e} 2b_1b_2 + (2g - 2)b_1^2 \geq 2gb_0^2. \end{align}Now we apply $(C_0 \cdot Y(r,s)) < 0$ which gives \begin{align}\label{eq} b_1\sqrt{(g+s)r} + b_2\sqrt{\frac{g+s}{r}} < 2a_3gb_0. \end{align} Since $b_1, b_2 \geq 0$, the left hand side of (\ref{eq}) is nonnegative. Thus we can square (\ref{eq}) \footnote{This is the only step where $a_3$ makes an appearance and it is immediately being squared. The proof proceeds exactly as in \cite{vojta} from here.}  and combine it with (\ref{e}) to get $$(g+s)(b_2^2/r + 2b_1b_2 + b_1^2r)  < 4g(b_1b_2 + (g-1)b_1^2).$$
 Rearranging this, we get $$b_2^2(g+s)/r + 2b_1b_2(s-g)  + b_1^2((g+s)r - 4g(g-1)) < 0.$$ This is a quadratic form in $b_1, b_2$ and therefore its discriminant must be nonnegative. Solving for $r$ then gives $$r \leq \frac{(g+s)(g-1)}{s}.$$ However this contradicts the hypothesis about $r$. Hence no such $C_0$ can exist and $Y(r,s)$ must be nef. 
\end{proof}

\noindent We use Proposition \ref{vojta} to prove the following result. 
\begin{prop}\label{ray} If $\nu = e_2 + q\delta $ and $q \ne 0$ then $$\nu \notin \NE(C \times C).$$ \end{prop}

\begin{proof}If we pick $a_3$ so that $a_3q = |q|$, then \begin{align*} (Y(r,s) \cdot \nu) & =  a_1 - 2gqa_3 \\ &=  \sqrt{\frac{g+s}{r}} - 2|q|g.\end{align*}  Now letting $s= 1$ and $r$ tend to  $\infty$, we get that $\displaystyle \sqrt{\frac{g+s}{r}}$ approaches $0$. This forces $(Y(r,1) \cdot \nu)$ to approach $ -2|q|g < 0$, implying that for $ r \gg 0$, $(Y(r,s) \cdot \nu) < 0$. We conclude that $\nu$ is not pseudoeffective, since its intersection with a nef divisor is negative. 
\end{proof}

\noindent We are now ready to prove Theorem \ref{thm1}. 
\begin{proof}[Proof of Theorem \ref{thm1}] It suffices to apply Proposition \ref{crit} with $h =  \displaystyle \frac{e_1 + e_2}{2}$, $e = e_2$ and $f = \delta$. Proposition \ref{ray} tells us that condition (\ref{b}) in Proposition \ref{crit} is satisfied.\end{proof}

\begin{rem} Observe that for $C/k$, where $k = \bar{k}$ is a field of characteristic $p > 0$, Theorem \ref{thm1} is easily seen to be true because the graph of the $e^\text{th}$ power of Frobenius, denoted by $\Delta_e$, is irreducible and $(\Delta_e \cdot \Delta_e) < 0$. It follows that $\NE(C \times C)$ has infinitely many extremal rays, hence is not polyhedral.
\end{rem}
\bibliography{biblio}
\bibliographystyle{plain}
\end{document}